\documentclass[10pt]{article}
\usepackage{amssymb}
\usepackage{amsfonts}
\usepackage{color}
\usepackage{xy}
\xyoption{all}

\newtheorem{theorem}{Theorem}[section]
\newtheorem{proposition}[theorem]{Proposition}

\newtheorem{definition}[theorem]{Definition}

\newtheorem{example}[theorem]{Example}

\newtheorem{remark}[theorem]{Remark}

\newenvironment{proof}{{\noindent\bf Proof.}}{\hfill $\Box$\par\vskip3mm}

\newcommand{\Hom}{{\rm Hom}}

\newcommand{\Cc}{\mathcal{C}}
\newcommand{\Dd}{\mathcal{D}}

\newcommand{\Mm}{\mathcal{M}}

\def\-1{_{(-1)}}
\def\0{_{(0)}}
\def\1{_{(1)}}
\def\2{_{(2)}}
\def\3{_{(3)}}

\def\ZZ{{\mathbb Z}}

\begin{document}

\author{Miodrag Cristian Iovanov \small \\  Department of Algebra, University of Bucharest \\ Academiei 14, Bucharest, Romania \\and\\ State University of New York @ Buffalo, \\244 Mathematics Building, Buffalo, NY 14260 \vspace{.5cm} \\ Lars Kadison \small \\
Department of Mathematics, University of Pennsylvania \\ David Rittenhouse Lab, 209 S. 33rd St., Philadelphia, PA 19104 }

\title{When  weak Hopf algebras are Frobenius}

\maketitle

\begin{abstract}
We investigate when a weak Hopf algebra $H$ is Frobenius; we show this is not always true, but it is true if the semisimple base algebra $A$ has all its matrix blocks of the same dimension. However, if $A$ is a semisimple algebra not having this property, there is a weak Hopf algebra $H$ with base $A$ which is not Frobenius (and consequently, it is not Frobenius ``over'' $A$ either). We give, moreover, a categorical counterpart of the result that a Hopf algebra is a Frobenius algebra for a noncoassociative generalization of  weak Hopf algebra.
\footnote{The first author was partially supported by the contract nr. 24/28.09.07 with UEFISCU "Groups, quantum groups, corings and representation theory" of CNCIS, PN II (ID\_1002)
}
\end{abstract}

\date{}


\section*{Introduction}

Quasi-Hopf algebras are objects generalizing Hopf algebras, which were introduced in 1990 by Drinfeld. They are associative algebras $H$, having also a coalgebra structure which is only coassociative up to conjugation by a nonabelian 3-cocycle and, together with an antipode $S:H\rightarrow H$ satisfy appropriate compatibility conditions. Weak Hopf algebras were  introduced and  investigated by several authors (\cite{BoSz, BNS, N, S})
and are of great interest in physics (e.g. \cite{BoSz, FFRS}). 

Finite dimensional Hopf algebras generalize group algebras of finite groups in many ways, and one of the interesting similar features they have is that they are Frobenius algebras. This property is also preserved in the infinite dimensional case, in the sense that an infinite dimensional Hopf algebra having a nondegenerate integral is co-Frobenius (as a coalgebra). It is then natural to investigate whether the property of being a Frobenius algebra is preserved for generalizations of Hopf algebras. This was shown to be true for finite dimensional quasi-Hopf algebras; even infinite dimensional co-quasi Hopf algebras are co-Frobenius as shown in \cite{BC}.

 Also the paper \cite{V} seems to show that  finite dimensional weak Hopf algebras are Frobenius algebras. This was done there as a consequence of the integral theory the author develops for weak Hopf algebras. However, G.~B\"ohm noticed a gap in one of the key ingredients used for proving this fact on \cite[p. 485]{V}. But it remains true that weak Hopf algebras are quasi-Frobenius, which was proved in \cite{V} and in \cite{BNS}. 

\vspace{.5cm}

In this note, we address this question by using the categorical approach and the more general language of finite tensor categories. This seems an appropriate approach, since the quasi-Frobenius property of an algebra is a Morita invariant property, and the Frobenius property is easily understood as well in terms of dimensions inside the category. This also allows including all the above mentioned finite dimensional structures (Hopf algebras, quasi-Hopf algebras) and addresses also the more general weak quasi-Hopf algebras. We show that a (finite dimensional) weak (or weak quasi) Hopf algebra $H$ is Frobenius if the dimensions of the matrix blocks of its base algebra $A$ are all equal. In particular, this recovers the well known results for Hopf and quasi-Hopf algebras. We also show that this is the best possible result, for if $A$ is a  separable algebra over an algebraically closed field, and the dimensions of the matrix blocks of $A$ are not all equal, then there is a weak Hopf algebra with base algebra $A$ which is not Frobenius. This is shown by constructing a tensor category $\Cc$ together with a tensor functor $F:\Cc\rightarrow {\rm Bimod}(A)$ into the category of $A$-bimodules (where $A$ is semisimple); then, applying general Tannakian reconstruction theory, we find the existence of a corresponding weak Hopf algebra which is not Frobenius.


\section{Preliminaries}

Throughout this note, $K$ will be an algebraically closed field of characteristic $0$. We recall the following definition, for example, from \cite{BoSz}, 

\begin{definition}
A weak Hopf algebra is a $K$-vector space $H$ which is both an associative algebra $(H,m,1)$ and a coassociative coalgebra $(H,\Delta,\varepsilon)$ together with an antipode $S:H\rightarrow H$ such that $\Delta:H\rightarrow H\otimes_K H$ is a morphism of algebras and\\
(i) $\varepsilon(ab\1)\varepsilon(b\2c)=\varepsilon(abc)=\varepsilon(ab\2)\varepsilon(b\1c)$;\\
(ii) $1\1\otimes_K 1\2{1'}\1 \otimes {1'}\2 =1\1 \otimes 1\2 \otimes 1\3= 1\1 \otimes {1'}\1 1\2 \otimes {1'}\2$, where $1'=1$ is a ``copy'' of 1;\\
(iii) $a\1S(a\2)=\varepsilon(1\1a)1\2$;\\
(iv) $S(a\1)a\2 = 1\1\varepsilon(a1\2)$;\\
(v) $S(a\1)a\2S(a\3)=S(a)$.
\end{definition}

\vspace{.5cm}

If the counit is an algebra homomorphism,
it is in fact just a Hopf algebra.  

Also recall the definition of quasi-Hopf algebra \cite{D};
examples of which come from bialgebras with their comultiplication conjugated by a gauge element \cite{D}.
\begin{definition}A quasi-bialgebra $H$ is an
algebra with additional 
structure $(H, \Delta, \varepsilon, \Phi)$ where  $\Delta: H \rightarrow H \otimes H$  is an algebra homomorphism
and noncoassociative coproduct
with counit augmentation $\varepsilon: H \rightarrow k$ satisfying the ordinary counit laws, $(\varepsilon \otimes \mbox{Id})\Delta= \mbox{Id} = (\mbox{Id} \otimes \varepsilon)\Delta.$
$\Phi$ is an invertible element in $H \otimes H \otimes H$ denoted by 
$\Phi = X^1 \otimes X^2 \otimes X^3 $
with inverse denoted by
$\Phi^{-1} = x^1 \otimes x^2 \otimes x^3 $
which controls the noncoassocativity of the coproduct on $H$ as follows:
\begin{equation}
\label{eq: quasi-coassoc}
(\mbox{Id} \otimes \Delta)(\Delta(a)) = \Phi (\Delta\otimes \mbox{Id})(\Delta(a)) \Phi^{-1}
\end{equation}
Moreover $\Phi$ must satisfy normalized $3$-cocycle equations given by: 
\begin{equation}
\label{eq: 3-cocycle}
(1 \otimes \Phi)(\mbox{Id} \otimes \Delta\otimes \mbox{Id})(\Phi)(\Phi \otimes 1) = 
(\mbox{Id} \otimes \mbox{Id} \otimes \Delta)(\Phi)(\Delta\otimes \mbox{Id}  \otimes \mbox{Id})(\Phi)
\end{equation}
\begin{equation}
\label{eq: normal}
(\mbox{Id} \otimes \varepsilon \otimes \mbox{Id})(\Phi) = 1 \otimes 1.
\end{equation}
A quasi-bialgebra $H$ is  a \textit{quasi-Hopf algebra} if there is an anti-automorphism $S: H \to H$, called
an \textit{antipode}, with elements
$\alpha, \beta \in H$ such that $
S(a\1) \alpha a\2  =  \varepsilon(a) \alpha$,
$a\1 \beta S(a\2)    =  \varepsilon(a) \beta$ for all $a \in A$; moreover, $X^1 \beta S(X^2) \alpha X^3 =  1$, and  
$S(x^1) \alpha x^2 \beta S(x^3)  =  1$
\end{definition} 

It was noted in \cite{BoSz} that the category of finite dimensional representations of a weak Hopf algebra $H$ is a monoidal category, and moreover, it has a rigid structure, that is, it has left and right duals of every object, satisfying appropriate axioms. This has motivated the introduction of finite tensor categories. We recall the following definitions; we refer the reader to \cite{BK} or \cite{CE} for technical details. 

\begin{definition}
(i) An abelian category $\Cc$ is called {\it finite} if there are only finitely many isomorphism types of simples, each of which has a projective cover,  objects have finite length and the Hom-spaces are finite dimensional. (Note: This is known to be equivalent to $\Cc\simeq {\rm Rep} A$ for a finite dimensional algebra $A$).\\
(ii) A category $\Cc$, which is $K$-linear abelian monoidal, rigid, all objects have finite length and the Hom spaces are finite dimensional vector spaces over $K$ is called a finite multitensor category; if ${\mathbf 1}$ is simple then this is called a finite tensor category. \\
(iii) A functor $F:\Cc\rightarrow \Dd$ between two monoidal categories is called quasi-tensor if there exist a natural isomorphism $\theta:F(X\otimes Y)\cong F(X)\otimes F(Y)$ for all $X,Y\in\Cc$, and $F(1_\Cc)\cong1_\Dd$. This is called a tensor functor if the isomorphism $\theta$ agrees with the associativity and unit isomorphisms of $\Cc$ and $\Dd$. 
\end{definition}

\vspace{.5cm}

It has been proven in \cite{EO} (see also \cite{HO}) that any finite tensor category is equivalent to the representation theory of a weak quasi-Hopf algebra: we refer to \cite{MS} for the technically precise definition of this generalization of weak Hopf algebra
and quasi-Hopf algebra not needed in this paper (although  the counit is not required to be an algebra homomorphism). This is done as follows: one first constructs a quasi-tensor functor $F:\Cc\rightarrow {\rm Bimod}(A)$ to the category of finite dimensional $A$-bimodules of a semisimple algebra $A$. Then, it is shown that $\Cc$ is equivalent to the representation category of a certain weak quasi-Hopf algebra structure on ${\rm End}_K(F)$
(which is a weak Hopf algebra if $F$ is a tensor functor). \\
Conversely, if $\Cc={\rm Rep}\,H$ - the finite dimensional representations (left modules) of $H$ - for a weak quasi-Hopf algebra $H$, one defines $A=A_L=\{ \varepsilon(1\1a)1\2, a \in A \}$ and $A_R= \{ 1\1\varepsilon(a1\2), a \in A \}$; it turns out that $A_R\simeq (A_L)^{op}$ via the antipode, $A_L$ and $A_R$ are semisimple and commute with each other; the algebra $A$ is called the base algebra of $H$. (The unit object
$A$ in $H$-Mod, with action via projection formula implicit above, is simple, or irreducible, in many cases; e.g. for $H$ a connected groupoid algebra.) There exists a functor $F:\Cc\rightarrow A\otimes A^{op}\!-\!\mbox{Mod}=\mbox{\rm Bimod}(A)$, which is a quasi-tensor functor. Moreover, if the algebra $H$ is a weak Hopf algebra, then $F$ is a tensor functor.

\vspace{.5cm}

In order to distinguish from the categorical duals coming from the rigidity axioms of a tensor category, we will use $V^\vee$ to denote the $K$-dual of the $K$-vector space $V$. 

\begin{definition} 
A finite dimensional algebra is called quasi-Frobenius if every (finitely generated) projective left (or, equivalently, right) $A$-module is injective, equivalently $A$ is an injective left (or right) $A$-module. A $K$-algebra is called Frobenius if and only if $A\simeq A^\vee$ as left, (or, equivalently, as right) $A$-modules.
\end{definition}

It has been shown in \cite[Corollary 3.3]{V} that a weak Hopf algebra is necessarily quasi-Frobenius; this result being obtained as a consequence of the integral theory and a Hopf module isomorphism for weak Hopf modules. This is also an easy consequence of \cite[Proposition 2.3]{EO}, which provides a very short proof of this fact. Moreover, this proposition also implies the result for weak quasi-Hopf algebras. We use the tensor category philosophy to obtain results about the Frobenius property of weak Hopf algebras and weak quasi-Hopf algebras. We note that since $K$ is algebraically closed, a finite dimensional quasi-Frobenius algebra $A$ is Frobenius if and only if the socle and co-socle of any projective indecomposable $A$-module $P$ have the same dimension \cite[16.7-16.33]{L}. Indeed, the multiplicity of $P$ in ${}_AA$ is the same as the multiplicity of $S$, the co-socle
$P/J(A)P$ of $P$, in $A/J(A)$ - the semisimple residual algebra of $A$. This is exactly $dim_K(S)$ since $K$ is algebraically closed ($S$ is simple). The multiplicity of $P$ in $A^\vee$ is the same as the multiplicity of $Q=P^\vee$ as right $A$-module in $(A^\vee)^\vee=A_A$. Since $P$ is also injective, $Q=P^\vee$ is projective as right $A$-module, and its multiplicity in $A_{A}$ is the dimension of the co-socle $T=Q/QJ(A)$ of $Q$. But $T^\vee$ is then the socle of $P$ by duality, and we get that the multiplicity of $P$ in $A$ and $A^\vee$ are the same iff ${\rm dim}(T)={\rm dim}(S)$.

\vspace{.5cm}
\section{Weak Hopf Algebras and Frobenius properties}

Let $H$ be a weak quasi-Hopf algebra, and $F:\Cc= {\rm Rep}(H)\rightarrow \mbox{\rm Bimod}(A)$ the associated forgetful functor. Let $(V_i)_{i\in I}$ be the simple objects in $\Cc$ and $P_i$ their respective covers. The vector space dimension of an $H$-module $M$ is $\dim_K(F(M))$ (this is different from the categorical or Frobenius-Perron dimension). Then $H=\bigoplus\limits_{i\in I}P_i^{\dim(F(V_i))}$. Also denote by $(S_j)_{j=1,\ldots,p}$ the simple right $A$-modules; then $S_{ij}=S_i^\vee\otimes_K S_j$ are the simple $A$-bimodules. Let $d_i=\dim_K(S_i)$. We recall that each object $X$ of $\Cc$ has an associated matrix $N_{Xj}^k$ defined by the left multiplication by $X$, where $N_{Xj}^k$ is the multiplicity of $V_k$ in the Jordan-H$\rm\ddot{o}$lder series of $X\otimes L_j$ in $\Cc$. As in \cite{EO}, Section 2.8, for each projective $P_i$, $i\in I$ let $D(i)\in I$ be such that $P_i^*\simeq P_{D(i)}$ (here $(-)^*$ denotes the categorical right dual). Also, there is an invertible object $V_\rho$ of $\Cc$ such that $P_{D(i)}=P_{{}^*i}\otimes V_\rho$ and $V_{D(i)}=V_{{}^*i}\otimes V_\rho={}^*V_i\otimes V_\rho$, where we convey $V_{{}^*i}={}^*V_i$. 

\begin{proposition}\label{p1}
Denote $[F(X):S_{ij}]$ the multiplicity of $S_{ij}$. Then
\begin{eqnarray*}
\dim_K(F(soc(P_k))) & = & \sum\limits_{i,j}[F(V_{D(k)}):S_{ij}]d_id_j\\
\dim_K(F(cosoc(P_k))) & = & \sum\limits_{i,j}[F({}^*V_k):S_{ij}]d_id_j
\end{eqnarray*}
\end{proposition}
\begin{proof}
We have $P_k^*\rightarrow cosoc(P_k^*)\rightarrow 0$, equivalently, by taking left duals, we get $0\rightarrow {}^*cosoc(P_k^*)\rightarrow {}^*(P_k^*)=P_k$ so $soc(P_k)={}^*cosoc(P_k^*)$. Also, $\dim_K(F({}^*X))=\dim_K(F(X^*))=\dim_K(F(X))^\vee=\dim_K(F(X))$ (in Bimod($A$) left and right duals are the same). Therefore
\begin{eqnarray*}
\dim_K(F(soc(P_k))) & = & \dim_K(F({}^*cosoc(P_k^*)))\,\,\,\,\,\,\,\,\,({\rm by\,duality})\\
& = & \dim_K(F({}^*cosoc(P_{D(k)})))\,\,\,\,\,\,\,\,\,(P_{D(k)}\simeq P_k^*)\\
& = & \dim_K(F({}^*V_{D(k)}))=\dim_K(F(V_{D(k)}))\\
& = & \sum\limits_{i,j}[F(V_{D(k)}):S_{ij}]d_id_j
\end{eqnarray*}
The second equality follows similarly.
\end{proof}

If $X,Y$ are objects of $\Cc$, then the matrix $M_X=[F(X):S_{ij}]_{i,j=1,\ldots,n}$ has integer coefficients, and moreover, $M_{X\otimes Y}=M_XM_Y$. Indeed, if $F(X)=\bigoplus\limits_{i,j}\alpha_{ij}S_{ij}$ and $F(Y)=\bigoplus\limits_{i,j}\beta_{ij}S_{ij}$, since $S_{ij}\otimes_AS_{kl}=\delta_{jk}S_{il}$ we get $F(X)\otimes_AF(Y)=\bigoplus\limits_{i,j,k,l}\delta_{jk}\alpha_{ij}\beta_{kl}S_{ij}S_{kl}=\bigoplus\limits_{i,l}(\sum\limits_{k=1}^p\alpha_{ik}\beta_{kl})S_{il}$. With this we have:

\begin{theorem}
Let $H$ be a finite dimensional weak quasi-Hopf algebra with the base algebra $A$. If the dimensions of the simple components of $A$ are all equal, then $H$ is a Frobenius algebra. In particular, this is true if the base algebra $A$ is commutative, so also when $H$ is a quasi-Hopf algebra. 
\end{theorem}
\begin{proof}
Since $V_{D(k)}={}^*V_k\otimes V_\rho$, $M_{V_{D(k)}}=M_{{}^*V_k}\cdot M_{V_\rho}$. But since $V_\rho$ is invertible, $M_{V_\rho^{\,}}\cdot M_{V_\rho^{-1}}=M_{V_\rho^{-1}}\cdot M_{V_\rho^{\,}}=M_{V_\rho\otimes V_\rho^{-1}}=M_{\mathbf 1}={\rm Id}$, so $M_{V_\rho}$ is a permutation matrix, since it has integer coefficients and its inverse $M_{V_\rho^{-1}}$ has integer coefficients too. So the columns of, and the elements of $M_{V_{D(k)}}=[F(V_{D(k)}):S_{ij}]_{i,j=1,p}$ are a permutation of the columns of, and respectively the elements of $M_{{}^*V_k}=[F({}^*V_k):S_{ij}]$. Thus, if $d=d_i=d_j$ for all $i,j$ (for commutative $A$, $d=1$), using the previous Proposition, one has $$\dim_K(F(soc(P_k)))=d^2\sum\limits_{i,j}[F(V_{D(k)}):S_{ij}]=d^2\sum\limits_{i,j}[F({}^*V_k):S_{ij}]=\dim_K(F(cosoc(P_k)))$$
and so $H$ is Frobenius.
\end{proof}

\begin{remark}
For the above proposition, we do not need that the characteristic of $K$ is $0$, since for any weak quasi-Hopf algebra, one can build a forgetful functor $F:{\rm Rep}(H)\rightarrow \mbox{\rm Bimod}(A)$.
\end{remark}

\begin{example}
\begin{rm}
Let $B$ be Taft's Hopf algebra of dimension $p^2$, with generators $g,x$ with $g^p=1$, $x^p=0$, $xg=\lambda gx$ with $\lambda$ a primitive $p$'th root of unity, and comultiplication $\Delta(g)=g\otimes g$, $\Delta(x)=g\otimes x+x\otimes 1$, counit $\varepsilon(g)=1$, $\varepsilon(x)=0$ and antipode $S(g)=g^{-1}$, $S(x)=-g^{-1}x$. Let $x_k^{(s)}=\sum\limits_{i=0}^{p-1}\lambda^{-ik}g^ix^s$, where we agree to write all indices modulo $p$. Note that 
\begin{eqnarray}
g\cdot x_k^{(s)} & = & \lambda^kx_k^{(s)} \label{eq1}\\
x\cdot x_k^{(s)} & = & x_{k-1}^{(s+1)} \label{eq2}
\end{eqnarray}
Denote $V_k$ the 1-dimensional $B$-module $K$ with structure $x\cdot \alpha=0$ and $g\cdot \alpha=\lambda^k\alpha$. These form a set of representatives for simple left $B$-modules. \\
Let $I_k^i=B\cdot x_{k}^{i}={}_K<x_{k+1}^{(p-1)},x_{k+2}^{p-2},\dots,x_{k+i+1}^{(p-i)}>$ ($i=1,\dots,p$) - the $K$-subspace spanned by these $i$ elements, which are linearly independent (since they contain different powers of $x$). There is a Jordan-H\"older series of $I_k^k$, $0=\subset I_k^1\subseteq I_k^2\subseteq\dots\subseteq I_k^{k-1}\subseteq I_k^k$ and the terms of these series are $I_k^i/I_k^{i-1}\simeq V_{k+i}$ by (\ref{eq1}). We have $J(B)=B\cdot x$ - the Jacobson radical of $B$. Then $J(B)\cdot I_k^k=I_k^{k-1}$ is the Jacobson radical of $I_k^k$, and so we have a superfluous morphism $I_k^k\rightarrow V_k\rightarrow 0$ (i.e. an epimorphism whose kernel is small). If $P_k$ is the projective cover of $V_k$, it follows that $P_k$ projects onto $I_k^k$, so it has dimension at least $p$. Since each projective indecomposable $P_k$ occurs in a decomposition of ${}_BB$, we have that $p^2=\dim_K(B)\geq\sum\limits_{i=0}^{p-1}\dim_K(P_i)\geq p\cdot p$, therefore $\dim(P_k)=p$ for all $k$ and thus $P_k=I_k^k$. In fact, since $I_k^{i-1}=J(B)\cdot I_k^i$, we can see that each $P_k$ is a chain module (the $I_k^i$ are the only submodules). However, we $P_k$ is also injective and has simple socle, and $soc(P_k)=I_k^1\simeq V_{k+1}$. We now build a tensor functor $F:{\rm Rep}(B)\rightarrow {\rm Bimod}(A)$ in several steps.
$$\xymatrix{
{\rm Rep}(B) \ar[d]^{F} \ar[r]^{F_1} &  {\rm Rep}(\ZZ/p)\ar[d]_{F_2} \\
{\rm Bimod}(A) & {\rm Bimod}(K[\ZZ/p])\ar[l]_{F_3}
}$$
First, let $F_1:{\rm Rep}(B)\rightarrow {\rm Rep}(\ZZ/p)$ be the forgetful functor, given by the inclusion of Hopf algebras  $<1,g,...,g^{p-1}>\simeq K[\ZZ/p]\hookrightarrow B$. It is easily checked that this is a tensor functor. Let $F_2:K[\ZZ/p]\!-\!\mbox{\rm mod}={\rm Rep}(\ZZ/p) \rightarrow {\rm Bimod}(K[\ZZ/p])$, $F_2(V_k)=\bigoplus\limits_{i+j=k}V_i^*\otimes_K V_j=\bigoplus\limits_i V_{-i}\otimes_K(V_{-i}\otimes V_k)$ where the second tensor represents the tensor product in ${\rm Rep}(\ZZ/p)$, and $(-)^*$ is the dual in ${\rm Rep}(\ZZ/p)$. On morphisms $f:X\rightarrow Y$, we have $F_2(f)=\bigoplus\limits_{i=0}^{p-1}1_{V_i^*}\otimes_K(1_{V_{-i}}\otimes f)$. It can be easily noted that $F_2$ is well defined (the action of $K[\ZZ/p]\otimes K[\ZZ/p]$ on $V_i\otimes_KV_j$ is on components) and that $F_2$ is a tensor functor: indeed, it is enough to check this on simple objects $F_2(V_a)\otimes_{K[\ZZ/p]}F_2(V_b)=(\bigoplus\limits_{i+j=a}V_{-i}\otimes V_j)\otimes_{K[\ZZ/p]}(\bigoplus\limits_{k+l=b}V_{-k}\otimes V_l)=\bigoplus\limits_{i,j,k,l;i+j=a,k+l=b}\delta_{j,-k}V_{-i}\otimes V_l=\bigoplus\limits_{i+l=a+b}V_{-i}\otimes V_l=F_2(V_a\otimes V_b)$.\\ Note that this functor can also be seen as the left adjoint of the functor 
$$G:{\rm Bimod(\ZZ/p)={\rm Rep}(\ZZ/p\times \ZZ/p)}\longrightarrow K[\ZZ/p]\!-\!\mbox{\rm mod}={\rm Rep}(\ZZ/p)$$ which is induced by the diagonal map $K[Z/p] \rightarrow K[Z/p]\otimes K[Z/p]$ coming by the group morphism $\ZZ/p\ni i\mapsto (-i,i)\in \ZZ/p\times\ZZ/p$. Indeed, we can easily see that $G(V_a^*\otimes_KV_b)=V_{a+b}$, so $\Hom_{{\rm Bimod}(\ZZ/p)}(F_2(V_k),V_a^*\otimes_KV_b)=\Hom_{{\rm Bimod}(\ZZ/p)}(\bigoplus\limits_{i+j=k}V_i^*\otimes_KV_j,V_a^*\otimes_KV_b)=\bigoplus\limits_{i+j=k}\delta_{i,a}\delta_{j,b}K=\delta_{a+b,k}K=\Hom_{K[\ZZ/p]-{\rm mod}}(V_k,V_{a+b})=\Hom_{K[\ZZ/p]-{\rm mod}}(V_k,G_{V_a^*\otimes_KV_b})$.\\
 Finally, let $d_1,\dots,d_p$ be positive integers and $A=\bigoplus\limits_{i=1}^pM_{d_i}(K)$ and $F_3: {\rm Bimod}(\ZZ/p)\rightarrow {\rm Bimod}(A)$, $F_3(V_i^*\otimes V_j)=S_i^\vee\otimes S_j=S_{ij}$ (with $S_i$'s as before). This is actually an equivalence of tensor categories. 
 Let $F=F_3\circ F_2\circ F_1:{\rm Rep}(B)\rightarrow {\rm Bimod}(A)$, which is a tensor functor. By the above, using Tannakian reconstruction, this corresponds to a weak Hopf algebra $H= \mbox{\rm Taft}(d_1,\dots,d_n)$ with base $A$ and ${\rm Rep}(H)\simeq {\rm Rep}(B)$, and ``forgetful'' functor the $F$ above. This holds in characteristic different from $0$, whenever none of the $d_i$ are divisible by the characteristic of $K$.
\end{rm}
\end{example}

\begin{proposition}
With the  notations above, the weak Hopf algebra $H$ is a Frobenius algebra if and only if $d_1,\dots,d_p$ are all equal. Also, the algebra $H$ has dimension $(\sum\limits_id_i)^4$. Thus, if the $d_i$'s are not all equal, $H$ is a weak Hopf algebra which is not a Frobenius algebra.
\end{proposition}
\begin{proof}
By the considerations above we have $\dim_K(F(soc(P_k)))=\dim_K(F(V_{k+1}))=$ \newline
$\dim_K(\bigoplus\limits_{i+j=k+1}S_i^\vee\otimes S_j)=\sum\limits_{i+j=k+1}d_id_j$ and also $\dim_K(F(cosoc(P_k)))=\dim_K(F(V_k))=\sum\limits_{i+j=k}d_id_j$. $H$ is Frobenius if and only if these two numbers are equal for all $k$. Let $\omega$ be a root of order $p$ of unity different from 1, and $t(x)=\sum\limits_{k=0}^{p-1}d_kx^k$. Then $t(\omega)^2=\sum\limits_{i,j}d_id_j\omega^{i+j}=\sum\limits_{k=0}^{p-1}\sum\limits_{i+j=k}d_id_j\omega^k=(\sum\limits_id_id_{-i})\cdot(\sum\limits_{k}\omega^k)=0$ (the indices are always considered mod $p$). Therefore, $t$ is divisible by the polynomial $\sum\limits_{k=0}^{p-1}x^p$ and so they are a multiple of each other. This implies that all $d_i$ are equal. \\
Since every projective indecomposable $P_k$ has each simple object occurring exactly once in any of its Jordan-H\"older series, $\dim_K(H)=\sum\limits_{k}\dim_K(F(P_k))
\cdot\dim_K(F(soc(P_k)))$
$=\sum\limits_k(\sum\limits_i\dim_K(F(V_i)))\cdot\dim_K(F(V_k))
=(\sum\limits_k\dim_k(F(V_k)))^2=
(\sum\limits_k\sum\limits_{i+j=k}d_id_j)^2$
$=(\sum\limits_kd_k)^4$.
\end{proof}

However, with the observation above on the characterization of Frobenius algebras, there is a certain categorical statement which could be interpreted as the analogue of the property of (quasi) Hopf algebras of being Frobenius:

\begin{proposition}
\label{prop-cat}
If $\Cc$ is a finite tensor category, then $d_+(soc(P_k))=d_+(cosoc(P_k))$, where $d_+$ represents the Frobenius-Perron dimension in $\Cc$.
\end{proposition}
\begin{proof}
As in Proposition \ref{p1}, $soc(P_k)={}^*L_{D(k)}$, so we compute  $d_+(soc(P_k))=d_+({}^*L_{D(k)})$
$=d_+(L_{D(k)})=d_+({}^*L_k\otimes L_\rho)=d_+({}^*L_k)d_+(L_\rho)=d_+(L_k)=d_+(cosoc(P_k))$.
\end{proof}

Note that proposition~\ref{prop-cat}  implies the Hausser-Nill result that quasi-Hopf algebras are Frobenius algebras. \\
One can ask whether a weak (quasi-) Hopf algebra $H$ is perhaps Frobenius ``over its base algebra'' $A$; this should naturally be formulated in the terminology of Frobenius extensions. If $\varphi:A\rightarrow B$ is a morphism (extension) of rings, it is called Frobenius extension if the forgetful (restriction of scalars) functor ${}_B\Mm\rightarrow {}_A\Mm$ has isomorphic left and right adjoints (see \cite{K} for details). We have the extension of algebras $A\hookrightarrow H$, and one can ask the question whether this is Frobenius. Since $A$ is semisimple, the unit $k\rightarrow A$ is a Frobenius extension, and then if $A\hookrightarrow H$ is Frobenius, $k\hookrightarrow H$ would be Frobenius (by transitivity of Frobenius extensions). But, as seen above, this is not always the case.
This provides an example of a (finite
projective) weak Hopf-Galois extension which is not a Frobenius extension, since $H$ is such
a Galois extension of $A$  \cite[2.7]{CDG}.  \\
 (Similarly, the extensions $A\hookrightarrow H$ and $A\otimes A^{op}= A_L\otimes A_R\hookrightarrow H$ are not twisted or $\beta$-Frobenius extensions, although they are
QF extensions since the base algebras are separable.)


\vspace{.5cm}

\begin{center}
\sc Acknowledgment
\end{center}

The first author wishes to thank Pavel Etingof from MIT for his great support and the excellent time spent at MIT during the fall semester of 2008. This paper would not have been possible without his contribution, as an important part of the ideas contained here are actually to his credit. The authors also thank Kornel Szlach\'anyi and colleagues in Budapest for comments about this paper.

\vspace*{3mm} 
\begin{flushright}
\begin{minipage}{148mm}\sc\footnotesize

Miodrag Cristian Iovanov\\
University of Bucharest, Faculty of Mathematics, Str.
Academiei 14,
RO-70109, Bucharest, Romania \&\\
State University of New York - Buffalo, 244 Mathematics Building, Buffalo NY, 14260-2900, USA\\
{\it E--mail address}: {\tt
yovanov@gmail.com, e-mail@yovanov.net}\vspace*{3mm}

Lars Kadison \\
Department of Mathematics, 
University of Pennsylvania, 
David Rittenhouse Laboratory,
209 S. 33rd St., Philadelphia,
PA 19104, USA \\
{\it E--mail address}: {\tt lkadison@math.upenn.edu}

\end{minipage}
\end{flushright}

\vspace{.5cm}

\begin{thebibliography}{J\c{S}}


\bibitem[BC]{BC} 
D. Bulacu, S. Caenepeel, Integrals for (dual) quasi-Hopf algebras. Applications, J. Algebra {\bf 266} (2003), no. 2, 552-583.

\bibitem[BNS]{BNS}
G. B\"ohm, F. Nill, K. Szlach\'anyi, \emph{Weak Hopf algebras I. Integral theory and C$^*$-structures}, J. Algebra {\bf 221} (1999), 385-438.

\bibitem[BoSz]{BoSz} 
G. B\"ohm, K. Szlach\'anyi, \emph{A coassociative C$^*$-quantum group with nonintegral dimensions}, Lett.Math. Phys. 38 (1996), no. 4, 437–456.

\bibitem[BK]{BK}
B. Bakalov, A. Kirillov, Jr. Lectures on tensor categories and modular functors, Univ.\ Lect.\ Ser.\ 21, Amer.\ Math. Soc., Providence, RI, 2001.

\bibitem[CE]{CE}
D. Calaque, P. Etingof, \emph{Lectures on Tensor Categories} in Quantum Groups (ed. B. Enriquez, V. Turaev), IRMA Lectures in Mathematics and Theoretical Physics 12, EMS, 2008.



\bibitem[CDG]{CDG} S. Caenepeel, E. De Groot, \emph{Galois theory for weak Hopf algebras}, Rev. Roumaine Math. Pures Appl.
{\bf 52} (2007), 151-176.
\bibitem[D]{D}{V.G. Drinfel'd,
Quasi-Hopf algebras, \textit{Leningrad Math.\ J.} \textbf{1} (1990), 1419--1457.}

\bibitem[EO]{EO}
P. Etingof, V. Ostrik, \emph{Finite Tensor Categories}, Mosc. Math. J., Vol. 4, No. 3 July-Sept 2004, 627-654.

\bibitem[ENO]{ENO}
P. Etingof, D. Nikshych, V. Ostrik, \emph{On fusion categories}, Ann. of Math. (2) 162 (2005), no. 2, 581--642.


\bibitem[FFRS]{FFRS}J.~Fr\"olich, J.~Fuchs, I.~Runkel, C.~Schweigert,
Picard groups in rational conformal field theory, \textit{Contemp.\ Math}
\textbf{391} (2005), 85--100.
\bibitem[HO]{HO}
R. Haring-Oldenburg, \emph{Reconstruction of weak quasi-Hopf algebras}, J. Algebra {\bf 194} (1997), no.1, 14-35.

\bibitem[K]{K}
L.~Kadison, New Examples of Frobenius Extensions, Univ.\ Lect.\
Ser. 14, Amer.\ Math.\ Soc., Providence, RI, 1999.

\bibitem[L]{L}
T.-Y.~Lam, Lectures on Modules and Rings, G.T.M. 189, Springer, 1999. 

\bibitem[MS]{MS}
G. Mack, V. Schomerus, \emph{Quasi-Hopf quantum symmetry in quantum theory}, Nuclear Phy. B {\bf 370} (1992), no. 1, 185-230.

\bibitem[N]{N}
D. Nikshych, \emph{On the structure of weak Hopf algebras}, Adv. Math. 170 (2002) 257-286.

\bibitem[S]{S}
K. Szlach\'anyi, \emph{Weak Hopf algebras}, Operator algebras and quantum field theory (Rome, 1996), 621--632, Int. Press, Cambridge, MA,
1997.


\bibitem[V]{V}
P. Vecserny\'es, \emph{Larson-Sweedler theorem and the role of grouplike elements in weak Hopf algebras}, J. Algebra {\bf 270} (2003), 471-520.



\end{thebibliography}
\end{document}